\documentclass[12pt,leqno,a4paper]{amsart}
\usepackage{amsmath,amssymb,amsthm,amsaddr}
\usepackage{enumitem}
\usepackage[margin=3cm,top=4cm,bottom=4cm,footskip=1cm,headsep=1.5cm]{geometry}
\usepackage[utf8]{inputenc}

\usepackage{amsthm}
\usepackage[british]{babel}

\author{H. Egger}
\address{Department of Mathematics, TU Darmstadt, Germany}
\email{egger@mathematik.tu-darmstadt.de}

\title[Galerkin approximation of Hamiltonian and gradient system]{Energy stable Galerkin approximation of Hamiltonian and gradient systems}

\newtheorem{theorem}{Theorem}

\theoremstyle{definition}
\newtheorem{remark}{Remark}

\def\C{\mathcal{C}}
\def\H{\mathcal{H}}
\def\L{\mathcal{L}}
\def\B{\mathcal{B}}

\def\u{u}
\def\v{v}
\def\f{f}
\def\g{g}

\def\RR{\mathbb{R}}
\def\VV{\mathbb{V}}

\def\dt{\partial_t}
\def\ddt{\frac{d}{dt}}
\def\dn{\partial_n}
\DeclareMathOperator{\curl}{curl}
\DeclareMathOperator{\divergence}{div}
\def\a{\mathbf{a}}
\def\b{\mathbf{b}}
\def\e{\mathbf{e}}
\def\j{\mathbf{j}}
\def\n{\mathbf{n}}
\def\w{\mathbf{w}}

\numberwithin{equation}{section}

\begin{document}

\begin{abstract}
A general framework for the numerical approximation of evolution 
problems is presented that allows to preserve exactly an underlying Hamiltonian- or gradient 
structure. The approach relies on rewriting the evolution problem in a particular form that 
complies with the underlying geometric structure. The Galerkin approximation of a
corresponding variational formulation in space then automatically preserves this 
structure which allows to deduce important properties for appropriate discretization schemes 
including projection based model order reduction.
We further show that the underlying structure is preserved also under
time discretization by a Petrov-Galerkin approach. 
The presented framework is rather general and allows the numerical approximation of a wide range 
of applications, including nonlinear partial differential equations and port-Hamiltonian systems. 
Some examples will be discussed for illustration of our theoretical results and connections 
to other discretization approaches will be revealed.
\end{abstract}

\maketitle

\vspace*{-1em}

\begin{quote}
\end{quote}

\begin{quote}
\noindent
{\small {\bf AMS-classification (2000):}
37K05, 
37L65, 
47J35, 
65J08  
}
\end{quote}

\section{Introduction} \label{sec:1}

The modeling of dynamical systems often leads to 
problems with a Hamiltonian or gradient structure which have been studied intensively in the literature. 
In this paper, we consider abstract evolution problems of the general form 
\begin{align} \label{eq:sys}
\C(\u) \dt \u &= -\H'(\u) + \f(\u), 
\end{align}
which include Hamiltonian and gradient systems as special cases.
Here $\f : \VV \to \VV'$ and $\C : \VV \to \L(\VV,\VV')$ are assumed to be at least continuous functions on some Banach space $\VV$ with dual $\VV'$, and $\L(\VV,\VV')$ denotes the space of linear bounded operators from $\VV$ to $\VV'$.
Moreover, $\H : \VV \to \RR$ is a continuously differentiable energy or storage functional with derivative $\H' : \VV \to \VV'$. 
For every point in time, equation \eqref{eq:sys} can therefore be understood in the sense of linear functionals in $\VV'$.
Any classical solution $u \in C^1(0,T;\VV)$ of \eqref{eq:sys} then satisfies 
\begin{align} \label{eq:energy}
\ddt \H(\u(t)) 
&= \langle \H'(\u(t)), \dt \u(t)\rangle \\
&= \langle \f(\u(t)), \dt \u(t)\rangle - \langle \C(\u(t)) \dt \u(t), \dt \u(t)\rangle  \notag
\end{align}
for all $0 \le t \le T$ where $\langle \cdot, \cdot \rangle$ denotes the duality product on $\VV' \times \VV$. In many cases, this \emph{energy identity} expresses the fact that the total energy $\H(\u)$ of the system changes in time only due to dissipation and the work done by external forces. 
Upon integration with respect to time, one can obtain a corresponding integral form
\begin{align} \label{eq:energy2}
\H(\u(t)) &= \H(\u(s)) + \int_s^t  \langle \f(\u(r)) - \langle \C(\u(r)) \dt \u(r), \dt \u(r)\rangle \dt \u(r)\rangle dr
\end{align}
which holds for all $0 \le s \le t \le T$ and which remains valid for less regular solutions solutions, e.g. $\u \in W^{1,p}(0,T;\VV)$. In that case, \eqref{eq:energy} 
is still valid for a.e. $0 < t < T$. 

Let us briefly mention two important cases that will be covered automatically by our results: 
(i) If $\C(u) = -\C(u)^*$ is skew-self adjoint and $\f(\u) \equiv 0$, then \eqref{eq:sys} models a Hamiltonian system and the energy $\H(\u)$ is preserved for all time. (ii) If $\C(u)=\C(\u)^*$ is positive definite and $\f(\u) \equiv \f \in \VV'$, then \eqref{eq:sys} is a gradient system and energy decays until a steady state is reached. 

Note that for the derivation of \eqref{eq:energy}, we only utilized the variational identity 
\begin{align} \label{eq:var}
\langle \C(\u(t)) \dt \u(t), \v\rangle &= -\langle \H'(\u(t)), \v\rangle  + \langle \f(\u(t)), \v\rangle  
\end{align}
for the special choice $\v = \dt \u(t)$. The validity of \eqref{eq:var} for all $\v \in \VV$ yields an equivalent variational formulation of the system \eqref{eq:sys} under consideration.

The energy identities \eqref{eq:energy} and \eqref{eq:energy2} play a fundamental role in the analysis of evolution problems \eqref{eq:sys} and they often encode important physical principles. 
Therefore, much research has been devoted to the construction and analysis of numerical methods that satisfy similar identities after discretization. 
In \cite{Gonzales96}, the concept of \emph{discrete derivative methods}, later often called \emph{discrete gradient methods}, was introduced 
which applied to \eqref{eq:sys} leads to time-stepping schemes of the form 
\begin{align} \label{eq:discrete_derivative}
\overline \C(\u^n,\u^{n-1}) \frac{\u^n - \u^{n-1}}{\tau} 
&= -\overline\H'(\u^n,\u^{n-1}) + \overline\f(\u^n,\u^{n-1}), \qquad n \ge 0. 
\end{align}
Particular approximations $\overline \C(\u^n,\u^{n-1})$, $\overline\H'(\u^n,\u^{n-1})$, $\overline\f(\u^n,\u^{n-1})$, i.e., the \emph{discrete derivative} and the \emph{average vector field} method, have been studied in \cite{Gonzales96,McLachlanEtAl99} and second order convergence with respect to the time step $\tau$ was established. Higher order extensions, i.e., the \emph{average vector field collocation} methods were proposed in \cite{CohenHairer11,Hairer10} for the numerical integration of Hamiltonian systems; their application to port-Hamiltonian systems was studied in \cite{CelledoniEtAl17}. 
The generalization to gradient and Hamiltonian systems on Riemannian manifolds was studied in \cite{HairerLubich13} and \cite{CelledoniEtAl18a,CelledoniEtAl18b}. 
Let us mention that the discrete gradient approach can also be utilized for the space discretization of nonlinear evolution problems; see \cite{MatsuoFurihata11} for details. 

In a recent paper \cite{Egger18}, we studied the systematic approximation of dissipative dynamical systems by means of Galerkin approximation in space and discontinuous Galerkin discretization in time. We here follow a similar route but and address problems of a different form \eqref{eq:sys} as well as their systematic discretization by Galerkin approximation in space and Petrov-Galerkin approximation in time.

As a first step of our analysis, we will show that the special structure of the problems under consideration is inherited automatically by Galerkin approximations of the corresponding variational formulation. Such approaches are frequently studied for discretization of partial differential equations \cite{ErnGuermond04,Thomee06} or in the context of model order reduction \cite{Antoulas05,BennerMehrmannSorensen05}. 
Let us emphasize that the structure preservation strongly depends on the particular form of 
the equation \eqref{eq:sys} and does, in general, not hold for other equivalent formulations 
of the evolution problem, like
\begin{align} \label{eq:syseq}
\dt \u  &= \B(\u) \H'(\u) + \g(\u),
\end{align}
which are frequently considered in the literature; see  \cite{Gonzales96,McLachlanEtAl99,Hairer10,CelledoniEtAl17} for instance. 
Note that \eqref{eq:syseq} can be transformed into \eqref{eq:sys} with $\C(\u) = \B(\u)^{-1}$ and $\f(\u) = \B(\u)^{-1} \g(\u)$, if $\B(\u)$ is invertible.
This may, however, not be the case after discretization. 

As a second step, we will show that corresponding energy identities also hold for time-discretization of \eqref{eq:sys} or \eqref{eq:var} by Petrov-Galerkin methods. Again, this observation is strongly based the particular structure of the evolution problems under consideration which should be taken into account in the modeling stage.
We will demonstrate by examples that this particular form arises quite naturally. 

We will further discuss the connection of our approach to the {\em discrete derivative} and the {\em average vector field collocation} methods, which can be viewed as special instances or inexact realizations of our methods.
The approach proposed in this paper, therefore, may provide further insight also into the analysis of these methods and the construction of new discretization schemes. 

\bigskip 

The remainder of the manuscript is organized as follows: 
In Section~\ref{sec:2}, we discuss the space discretization of the problem \eqref{eq:sys} by Galerkin approximation of the variational principle \eqref{eq:var}. In addition, we discuss inexact variants of the methods, which may be more convenient for a practical realization. 
In Section~\ref{sec:3}, we then study the time discretization by a Petrov-Galerkin approach and we highlight the connection to other methods that have been discussed in the literature. Again some level of inexactness 
is allowed that may facilitate the numerical treatment.
The applicability of our methods will be demonstrated in Section~\ref{sec:4}, 
where we discuss some typical test problems in finite and infinite dimensions.

\section{Space discretization} \label{sec:2}

Let $\VV_h \subset \VV$ be a closed sub-space of the state space $\VV$. 
We then consider the Galerkin approximation of the variational principle \eqref{eq:var}
in $\VV_h$ given by 
\begin{align} \label{eq:varh}
\langle \C(\u_h(t)) \dt \u_h(t), \v_h\rangle 
&= -\langle \H'(\u_h(t)), \v_h\rangle + \langle \f(\u_h(t)), \v_h\rangle 
\end{align}
which is assumed to hold for all $\v_h \in \VV_h$ and for all times $t$ relevant for the problem.
As a direct consequence of the particular structure of the system under consideration, we obtain the following rather general structure result. 
\begin{theorem} \label{thm:energyh}
Let $\u_h \in C^1([0,T];\VV_h)$ satisfy \eqref{eq:varh} for all $0 \le t \le T$ and $\v_h \in \VV_h$. 
Then   
\begin{align} \label{eq:energyh}
\ddt \H(\u_h(t)) &=  \langle \f(\u_h(t)),\dt \u_h(t)\rangle -\langle \C(\u_h(t)) \dt \u_h(t), \dt \u_h(t)\rangle.
\end{align}
As a consequence, one also has the integral identity
\begin{align*} 
\H(\u_h(t)) &= \H(\u_h(s))  + \int_s^t \langle \f(\u_h(r)), \dt \u_h(r)\rangle - \langle \C(\u_h(r)) \dt \u_h(r), \dt \u_h(r)\rangle \, dr
\end{align*}
for all $0 \le s \le t \le T$. The second identity remains valid for non-smooth solutions, e.g., if $\u_h \in H^1(0,T;\VV_h)$, while the first identity holds for a.e. $t$ in that case. 
\end{theorem}
\begin{proof}
The first identity follows by formal differentiation of $\H( \u(t))$ with respect to time and use of identity \eqref{eq:varh} with $\v_h = \dt \u_h(t)$; this is possible, since $\dt \u_h(t) \in \VV_h$ is an admissible test function. The second identity then simply follows by integration of the first identity with respect to time.
\end{proof}

\begin{remark} \label{rem:structure}
Let us emphasize that the result of the above theorem is a direct consequence of the particular form \eqref{eq:sys} of the evolution problem under consideration. In general, the argument does not apply, and 
the result does not hold, if the discretization is based on an equivalent reformulation of the problem; see e.g. \eqref{eq:syseq} and the corresponding remarks in the introduction.
\end{remark}

\begin{remark} \label{rem:inexact}
A similar result can be obtained, if reasonable approximations for the individual terms in the discrete variational problem \eqref{eq:varh} are used. 
As an example, let us consider an inexact Galerkin approximation of the form 
\begin{align} \label{eq:varhinex}
\langle \C_h(\u_h(t)) \dt \u_h(t), \v_h\rangle 
&= -\langle \H_h'(\u_h(t)), \v_h\rangle + \langle \f_h(\u_h(t)), \v_h\rangle.
\end{align}
Then we still obtain energy identities similar to \eqref{eq:energyh} or its integral form with $\H(\u)$, $\C(\u)$, and $\f(\u)$ replaced by $\H_h(\u_h)$, $\C_h(\u_h)$, and $\f_h(\u_h)$. One may even replace the duality product $\langle \cdot,\cdot \rangle$ on $\VV' \times \VV$ by another duality product $\langle \cdot,\cdot \rangle_h$ on $\VV_h' \times \VV_h$ and could even consider non-conforming Galerkin approximations with $\VV_h \not\subset \VV$; we refer to \cite{MatsuoFurihata11} for considerations in this direction and to Section~\ref{sec:4} for examples.
\end{remark}

\section{Time discretization} \label{sec:3}

We now turn to the time discretization of the variational problem \eqref{eq:var}.
Since this variational form of \eqref{eq:sys} is inherited by Galerkin approximation in space, 
the following arguments also cover problems that have already been discretized in space.

Let $T_N = \{0=t^0 < \ldots < t^N=T\}$ be a partition of $[0,T]$, set $\tau^n = t^n - t^{n-1}$, and denote by $P_k(T_N;\VV)$, $k \ge 0$ the space of piecewise polynomial functions over the partition $T_N$ with values in $\VV$. 
For the time discretization of \eqref{eq:var}, we consider the following Petrov-Galerkin approach: 
Find $u_N \in P_{k+1}(T_N;\VV) \cap H^1(0,T;\VV)$ satisfying
\begin{align} \label{eq:varN}
\int_{t^{n-1}}^{t^n} \langle \C(\u_N) \dt \u_N(t), \widetilde \v_N(t) \rangle dt 
&= \int_{t^{n-1}}^{t^n}  \langle \f(\u_N(t),\widetilde v_N(t)\rangle -\langle \H'(\u_N(t),\widetilde v_N(t)\rangle dt  
\end{align}
for all $\widetilde v_N \in P_{k}(T_N;\VV)$ and all time intervals $1 \le n \le N$.
Note that $u_N$ is a piecewise polynomial function of $t$ of degree $k+1$ and globally continuous, while $\widetilde v_N$ is a piecewise polynomial of degree $k$ and may be discontinuous at time points $t_n$, $n=1,\ldots,N-1$; therefore, \eqref{eq:varN} is a Petrov-Galerkin approximation. 

\begin{remark} \label{rem:various}
It suffices to consider scalar valued test functions $\widetilde \v_N \in P_k(T_h;\RR)$ in the discrete variational problem, in which case one obtains the equivalent formulation
\begin{align} \label{eq:varN2}
\int_{t^{n-1}}^{t^n} \C(\u_N) \dt \u_N(t) \widetilde \v_N(t) dt 
&= \int_{t^{n-1}}^{t^n}  \f(\u_N(t),\widetilde v_N(t) -\H'(\u_N(t),\widetilde v_N(t)  dt,
\end{align}
which then has to be understood as an equation in $\VV'$; compare with the original system \eqref{eq:sys} which is equivalent to the variational formulation \eqref{eq:var} on the continuous level. 
This will be useful for our discussions later on.
\end{remark}
With similar arguments as before, we now obtain the following energy identity.
\begin{theorem} \label{thm:energyN}
Let $\u_N \in P_{k+1}(T_N;\VV) \cap H^1(0,T;\VV)$ solve \eqref{eq:varN} for all $\widetilde \v_N \in P_k(T_N;\VV)$ and all $1 \le n \le N$ or, equivalently, \eqref{eq:varN2} for all $\widetilde \v_N \in P_k(T_N;\RR)$. Then 
\begin{align*} 
\H(\u_N(t^n)) = \H(\u_N(t^m)) + \int_{t^{n-1}}^{t^n} \langle \f(\u_N(t)),\dt \u_N(t)\rangle - \langle \C(\u_N(t)) \dt \u_N(t), \dt \u_N(t)\rangle dt
\end{align*}
for all time instances $0 \le t_m \le t_n \le T$ of the respective time grid $T_N$.
\end{theorem}
\begin{proof}
Using the fundamental theorem of calculus, we obtain
\begin{align*}
\H(\u_N(t^n)) 
&= \H(\u_N(t^{n-1})) + \int_{t^{n-1}}^{t^n} \langle \H'(\u_N(t)), \dt \u_N(t)\rangle \, dt = (i).
\end{align*}
By choice of the ansatz and test spaces, $\widetilde \v_N = \dt \u_N$ is an admissible test function for problem \eqref{eq:varN}, and we can replace the right hand side (i)  by 
\begin{align*}
(i) =  \int_{t^{n-1}}^{t^n} \langle \f(\u_N(t)),\dt \u_N(t)\rangle -\langle \C(\u_N(t)) \dt \u_N(t), \dt \u_N(t)\rangle dt.
\end{align*}
This proves the result for $m=n-1$ and the general case follows by induction.
\end{proof}

\begin{remark}
Similar to the space discretization, the proof and validity of the discrete energy identity strongly
relies on the particular structure of the problem \eqref{eq:sys} and its variational formulation \eqref{eq:var}. 
Let us further note that the energy identity here only holds at specific points in time and we do not have a pointwise energy identity like \eqref{eq:energy} or \eqref{eq:energyh} after time discretization. 
\end{remark}

Let us next comment on the connection to other approximation schemes that 
have been proposed for the time discretization of Hamiltonian and gradient systems.
\begin{remark} \label{rem:avf}
Consider the case $k=0$ in the Petrov-Galerkin method \eqref{eq:varN}. 
Then $\u_N$ is piecewise linear in time and $\widetilde \v_N$ is piecewise constant.
Using the abbreviation $\u^n = \u_N(t^n)$, the scheme \eqref{eq:varN} can be written equivalently as
\begin{align*}
\overline \C(\u^n,\u^{n-1}) \frac{\u^n - \u^{n-1}}{\tau^n} 
&= -\overline \H'(\u^n,\u^{n-1}) + \overline \f(\u^n,\u^{n-1}),
\end{align*}
with averages $\overline \C(\u^n,\u^{n-1}) = \frac{1}{\tau^n} \int_{t^{n-1}}^{t^n}\C(\u^n(\tau)) d\tau$, $\overline \H'(\u^n,\u^{n-1}) = \frac{1}{\tau^n} \int_{t^{n-1}}^{t^n} \H'(\u^n(\tau)) d\tau$, and $\overline\f(\u^n,\u^{n-1}) = \frac{1}{\tau^n} \int_{t^{n-1}}^{t^n} \f(\u^n(\tau)) d\tau$; here we used $\u^n(\tau)=\u^{n-1} + \tau (\u^{n} - \u^{n-1})$ for abbreviation.
For $k=0$, the method \eqref{eq:varN} thus coincides with a discrete gradient method outlined in the introduction; see \cite{Gonzales96,McLachlanEtAl99} and \cite{CelledoniEtAl17,HairerLubich13} for further details.
\end{remark}

\begin{remark} \label{rem:avf_collocation}
Similar as in the previous section, we can also allow for an inexact realization of the Petrov-Galerkin approximation without destroying the essential structure required for the proof of the energy identity.
Using a quadrature method to approximate the integral on the right hand side of \eqref{eq:varN2}, for instance, leads to
\begin{align} \label{eq:quad}
\int_{t^{n-1}}^{t^n} \C(\u_N(t)) \dt \u_N(t) \widetilde \v_N(t) \dt 
\approx \tau_n \sum_{j=0}^{k} w_j\C(\u_N(t^n_j)) \dt \u_N(t^n_j) \widetilde v_N(t^n_j),
\end{align}
with intermediate time points $t^n_j = t^{n-1} + \gamma_j (t^n - t^{n-1})$.
For convenience of notation, we here employed the equivalent formulation \eqref{eq:varN2} with scalar valued test functions $\widetilde \v_N \in P_k(T_N;\RR)$.
When choosing $\widetilde \v_N(t)$ as the Lagrange polynomials for the quadrature points $t^n_j$, this immediately leads to an \emph{average vector field collocation method}; see \cite{CelledoniEtAl17,Hairer10,HairerLubich13} for details.
If $\C(\u)=\C$ is independent of $\u$, then exact integration of the right hand side of \eqref{eq:varN2} is achieved by the Gau{\ss} quadrature rule. 
\end{remark}
The proposed Petrov-Galerkin approach can therefore be used to derive well-known time-discretization methods in a systematic manner and it provides a framework to generalize these methods to a wider class of problems.

\section{Examples} \label{sec:4}

We now illustrate the applicability of the proposed discretization approaches in space and time
by discussing some typical problems we have in mind.

\subsection{Magneto-quasistatics} \label{sec:num1}

As a first test problem, we consider the equations of magneto-quasistatics, which arise in the eddy current approximation of Maxwell's equations \cite{AlonsoValli10}. 
In this model, the magnetic flux density $\b=\curl \a$ is represented by a magnetic vector potential $\a$, which is assumed to satisfy 
\begin{align*}
\sigma \dt \a + \curl (\nu(\curl \a) \curl \a) &= -\j \qquad \text{in } \Omega.
\end{align*}
Here $\sigma$ denotes the electric conductivity, $ \nu=\mu^{-1}$ is inverse of the magnetic permittivity tensor $\mu$, and $\j$ a given source current density.
For ease of presentation, we assume that $\sigma$ is uniformly positive, in which case 
$\e=\dt \a$ then corresponds to the electric field density. 
Moreover, we consider homogeneous boundary conditions 
\begin{align*}
\a \times \n &= 0  \qquad \text{on } \partial\Omega,
\end{align*}
where $\n$ is the outward pointing unit normal vector on $\partial\Omega$.
The variational formulation for the above problem then reads 
\begin{align*}
\int_\Omega \sigma \dt \a(t) \cdot \w dx + \int_\Omega (\nu(\curl \a(t)) \cdot \curl \a(t)) \cdot \curl \w dx &= - \int_\Omega \j(t) \cdot \w dx,   
\end{align*}
which is supposed to hold for all $t$ of relevance and for all $\w \in H_0(\curl;\Omega)$ having a weak curl in $L^2(\Omega)$ and satisfying zero boundary conditions $\w \times \n=0$ on $\partial\Omega$; see~\cite{AlonsoValli10} for details.
We further define a scalar potential for the function $\nu(\b) \cdot \b$, i.e., 
\begin{align*}
h(\b) = \int_0^\b (\nu(\b) \cdot \b) \cdot d\b
\end{align*}
such that $\nabla_\b h(\b) = \nu(\b) \cdot \b$,
and recall that the expression 
\begin{align*}
\H(\a) = \int_\Omega h(\curl\a) dx, 
\end{align*}
then denotes magnetic energy of the system. 
By differentiation, we obtain
\begin{align*}
\ddt \H(\a(t)) 
&= \int_\Omega (\nu(\curl \a(t)) \cdot \curl \a(t)) \cdot \curl \dt \a(t) dx \\
&= -\int_\Omega \sigma \dt \a(t) \cdot \dt \a(t) dx - \int_\Omega \j(t) \cdot \dt \a(t) dx,
\end{align*}
where we used the variational formulation with test function $\w=\dt \a$ to perform the last step. 
This energy identity expresses the intuitive fact that the magnetic energy of the system is only altered due to dissipation caused by eddy currents and the work done by the excitation currents. 

Setting $\VV=H_0(\curl;\Omega)$, $\u = \a$, $\H(\u) = \int_\Omega h(\a) dx$, $\H'(\u) = \curl(\nu(\curl\a)) \curl a$, $\C(\u) \dt \u =\sigma \dt \a$,  and $\f(\u)=-\j$, one can see that the magneto-quasistatic problem perfectly fits into our abstract framework; note that, formally, the last three terms above have to be understood as linear functionals on $H_0(\curl;\Omega)$.

A Galerkin approximation of the variational principle in an appropriate finite element space $\VV_h \subset H_0(\curl;\Omega)$ thus leads to a semi-discretization which automatically inherits the energy identity derived above. After choice of a basis for the space $\VV_h$, the semi-discrete problem can be cast into a  system of ordinary differential equations
\begin{align*}
M_\sigma \dt a(t) + K_\nu(a(t)) a(t) &= -j(t).
\end{align*}
A further time discretization of this problem by a Petrov-Galerkin approximation, as proposed in Section~\ref{sec:3}, then leads to a fully discrete scheme which satisfies a corresponding energy 
identity in integral form and thus automatically incorporates the physical principle of energy conservation.

\begin{remark}
Before closing this section, let us briefly comment on some natural generalizations: Without any complications, one can consider other types of boundary conditions and a field dependent conductivity $\sigma(\e)$, where $\e = \dt \a$ denotes the electric field density. If $\sigma$ is allowed to vanish identically on a subdomain $\Omega_{nc} \subset \Omega$, then one has to restrict $\a$ in $\Omega_{nc}$ by some gauging conditions \cite{AlonsoValli10}; these can be treated, e.g., as additional constraints with similar arguments as in Section~\ref{sec:num3} below. 
\end{remark}

\subsection{Cahn-Hilliard equation} \label{sec:num2}

A simple model for the phase separation in binary fluids is given by the Cahn-Hilliard equation \cite{Elliott}
\begin{alignat*}{2}
\dt u &= -\Delta (\gamma \Delta u - \psi'(u)) \qquad &&\text{in } \Omega, \\
0 &= \dn u = \dn \Delta u \qquad && \text{on } \partial\Omega.
\end{alignat*}
Here $u$ represents the difference of the phase fractions of the two fluid components, $\psi$ is a double well potential with two minima in $[-1,1]$, and $\gamma>0$ is a positive constant. 
Using the homogeneous Neumann conditions $\dn u = 0$, one can verify that the 
integral $\int_\Omega u dx$ does not change over time, i.e, $ \ddt \int_\Omega u(t) dx = \int_\Omega \dt u(t) dx = 0$. Without loss of generality, we will further assume that $\int_\Omega u(t) dx = 0$. 

Let us denote by $(-\Delta_\gamma)^{-1}$ solution operator for the Neumann problem 
\begin{align*}
-\divergence(\gamma \nabla w) &= f \quad \text{in } \Omega, \qquad \gamma \dn w = 0 \quad \text{on } \partial\Omega,  
\end{align*}
which is has a unique solution $w \in H^1(\Omega)$ with zero average for any sufficiently regular right hand side $f$ with zero average. 
By formally applying this operator to the Cahn-Hilliard equation,
we obtain the simplified system
\begin{alignat*}{2}
(-\Delta_\gamma)^{-1}\dt u &= \Delta u - \psi'(u) \qquad &&\text{in } \Omega, \\
                  0 &= \dn u              \qquad &&\text{on } \partial\Omega,
\end{alignat*}
which will be the basis for our further considerations.
By testing the simplified problem with appropriate test functions $v$, we obtain the weak form 
\begin{align*}
((-\Delta_\gamma)^{-1}\dt u(t), v) &= -(\nabla u(t), \nabla v) - (\psi'(u(t)), v),
\end{align*}
of the evolution problem which is assumed to hold for any sufficiently regular test function $v$ 
and for any time $t$ under consideration.
Here $(a,b) = \int_\Omega a b \, dx$ was used to abbreviate the $L^2$-scalar product over $\Omega$. 

We will now show that, besides the conservation of mass, the solutions of the Cahn-Hilliard problem also describes the decay of the free energy
\begin{align*}
\H(\u) = \int_\Omega  \frac{1}{2} |\nabla u|^2 + \psi(u) dx.
\end{align*}
By inserting a solution and formally differentiating with respect to time, we obtain
\begin{align*}
\ddt \H(u(t)) 
= (\nabla u(t), \nabla \dt u(t)) + (\psi'(u(t)), \dt u(t) 
= -((-\Delta)^{-1} \dt u(t), \dt u(t)).
\end{align*}
For the second step, we used the variation principles with test function $v=\dt u(t)$.
This shows that the free energy is decreasing until the system reaches a steady state. 

A brief inspection of the above derivations shows that the Cahn-Hilliard problem in its simplified form has exactly the structure \eqref{eq:sys}, with $\VV = \{v \in H^1(\Omega) : \int_\Omega v dx = 0\}$, $\C(\u) = (-\Delta)^{-1}$ independent of $u$, and $\f(\u) \equiv 0$. 

A standard Galerkin approximation of the simplified form of the Cahn-Hilliard system in space would construct an approximation $u_h$ with values in $\VV_h \subset \VV$ satisfying a variational principle of the form 
\begin{align*} 
((-\Delta_\gamma^{-1} \dt u_h(t), v_h) &= -(\nabla u_h(t), \nabla v) - (\psi'(u_h(t)),v_h).
\end{align*}
While theoretically sound, such a method cannot be realized in practice, since the application of the inverse Laplacian $(-\Delta_\gamma)^{-1}$ can in general not be computed.
To overcome this problem, we utilize a discrete approximation 
$(-\Delta_{\gamma,h})^{-1}$, which is defined via the solution $\mu_h = (-\Delta_{\gamma,h})^{-1} \in \VV_h$ of 
\begin{align*}
(\gamma \nabla \mu_h, \nabla \eta_h) &= (f,\eta_h) \qquad \forall \eta_h \in \VV_h.
\end{align*}
Again, a zero average condition for the right hand side $f$ and the solution $\mu_h$ has to be imposed to guarantee existence of a unique solution. 
The discrete approximation $u_h$ for the Cahn-Hilliard system is then sought via the discrete variational principle
\begin{align*} 
((-\Delta_{\gamma,h})^{-1} \dt u_h(t), v_h) &= -(\nabla u_h(t), \nabla v) - (\psi'(u_h(t)),v_h),
\end{align*}
which is again required to hold for all $v_h \in \VV_h$ and all relevant $t$.
With similar reasoning as on the continuous level, one can verify the energy identity 
\begin{align*}
\ddt \H(\u_h(t)) &= -((-\Delta_h)^{-1} \dt u_h(t), \dt u_h(t)),
\end{align*}
which shows that also the energy of the discretized system decreases until a steady state is reached.
Let us note that the above method corresponds to an inexact Galerkin approximation in space, as discussed in Remark~\ref{rem:inexact}. 

For the sub-sequent time discretization, we can further employ a Petrov-Galerkin approximation as outlined in Section~\ref{sec:3}. Following our considerations in Remark~\ref{rem:avf_collocation}, this here coincides with a particular average vector field collocation method based on Gau\ss-quadrature \cite{HairerLubich13}. Other quadrature rules may, however, be used as well.

\subsection{Constrained Hamiltonian systems} \label{sec:num3}

As another class of applications, we now consider finite dimensional Hamiltonian systems with holonomic constraints 
\begin{align}
\dt q &= - H_p(p,q) \label{eq:h1}\\
\dt p &= H_q(p,q) + f(p,q) +  g_q(p,q)^\top \lambda \label{eq:h2}\\
    0 &= g(q).  \label{eq:h3a}
\end{align}
Such systems arise, for instance, in the modeling of multibody systems
but also in the context of electrical networks \cite{HairerLubichRoche89,KunkelMehrmann06}. 
Here $q,p$ are the vectors of generalized coordinates and momenta, 
$\H(p,q)$ is the Hamiltonian, i.e, the energy or storage functional,
$f(p,q)$ denotes the external forces, $\lambda$ is the vector of Lagrange multipliers for the constraints, 
and $\lambda^\top q_q(p,q)$ are the corresponding forces. 
We assume in the following that $p$, $q$, $\lambda$ are real valued vectors and use subscripts to denote partial derivatives. We then denote by $\langle \cdot,\cdot\rangle$ the Euclidean scalar product on $\RR^n$. 

By differentiating the constraint equation \eqref{eq:h3a} with respect to time, one obtains
\begin{align} \label{eq:h3}
0 &= g_q(q) \dt q, 
\end{align}
which is equivalent to $g(q)=0$ up to a constant factor that can be fixed by an appropriate initial conditions. 
The weak formulation of \eqref{eq:h1}--\eqref{eq:h3} here reads 
\begin{align}
\langle\dt q(t), v\rangle &= -\langle H_p(p(t),q(t)), v\rangle \label{eq:h1var}\\
\langle\dt p(t), w\rangle &=  \langle H_q(p(t),q(t),w\rangle + \langle f(p(t),q(t)),w\rangle + \langle g_q(q(t))^\top \lambda(t) ,w \rangle \label{eq:h2var}\\
0 &= \langle g_q(q(t)) \dt q(t),\eta\rangle  \label{eq:h3var}
\end{align}
These identities are again assumed to hold for all test functions $v$, $w$, $\eta$, and all $t$ of relevance.
For any smooth solution $(p,q,\lambda)$ of \eqref{eq:h1var}--\eqref{eq:h3var}, we then obtain 
\begin{align*}
\ddt H(p(t),q(t)) 
&= \langle H_p(p(t),q(t)),\dt p(t)\rangle + \langle H_q(p(t),q(t)),\dt q(t)\rangle \\
&= -\langle\dt q(t), \dt p(t)\rangle + \langle\dt p(t), \dt q(t)\rangle \\
& \qquad \qquad - \langle f(p(t),q(t)),\dt q(t)\rangle - \langle\lambda(t)^\top g_q(q(t)),\dt q(t)\rangle \\
&= - \langle f(p(t),q(t)),\dt q(t)\rangle.
\end{align*}
In the last step, we here used that $\langle g_q(q)^\top \lambda, \dt q\rangle = \langle g_q(q) \dt q,\lambda \rangle=0$ which follows from testing equation \eqref{eq:h3var} with $\eta=\lambda$.
This identity states that the energy $H(p,q)$ of the system can only change by 
the work of external forces. 

The system \eqref{eq:h1}--\eqref{eq:h2} and \eqref{eq:h3} and its variational formulation \eqref{eq:h1var}--\eqref{eq:h3var} are again of the abstract form \eqref{eq:sys}, \eqref{eq:var} with $\u=(p,q,\lambda)$, $\H(\u)=H(p,q)$, as well as
\begin{align*}
\C(\u)=\begin{pmatrix} 0 & I & 0 \\ -I & 0 & 0 \\ 0 & g_q(q) & 0 \end{pmatrix}
\qquad \text{and} \qquad
\f(\u)=\begin{pmatrix} 0 \\ -f(p,q) - g_q(q)^\top \lambda \\ 0 \end{pmatrix}.
\end{align*}
Hence all results about the approximation by Galerkin methods obtained in the previous sections can be applied immediately.

Any Galerkin projection of \eqref{eq:h1var}--\eqref{eq:h3var} into a subspace $\VV_h \subset \VV$ thus automatically inherits the energy identity stated above. Our results therefore cover 
general model order reduction approaches based on Galerkin projection \cite{BennerMehrmannSorensen05}.
The integral form of the energy identity also remains valid after time-discretization, if an appropriate Petrov-Galerkin approximation is used; see Theorem~\ref{thm:energyN} and Remark~\ref{rem:various}. 

\begin{remark}
Let us recall that in the variational formulation \eqref{eq:h1var}--\eqref{eq:h3var}, it suffices to 
test with scalar valued test functions; see Remark~\ref{rem:various}. 
Testing the linearized constraint \eqref{eq:h3var} with the particular test function $\eta = 1|_{[t^{n-1},t^n]}$ then yields
\begin{align*}
0 &= \int_{t^{n-1}}^{t^n} g_q(q(t)) \dt q(t) \eta(t) dt 
= \int_{t^{n-1}}^{t^n} \ddt g(q(t)) dt 
= g(q(t^{n})) - g(q(t^{n-1})). 
\end{align*}
The original constraint \eqref{eq:h3a} thus remains valid for all time points $t^n$, if it was valid at initial time and if the piecewise constant functions in time are elements of the test space $\widetilde V_N$ of the time-discretization \eqref{eq:varN}. The Petrov-Galerkin time discretization, therefore, does formally not suffer from the \emph{drift-off} phenomenon; see \cite{HairerLubichRoche89,KunkelMehrmann06}. 
\end{remark}

\section{Discussion}

In this paper, we presented an abstract framework for the numerical approximation of evolution problems with an underlying Hamiltonian- or gradient structure and we showed that this underlying structure is preserved under discretization with Galerkin methods in space and Petrov-Galerkin approximation in time. We further showed that some inexactness in the numerical realization of the Galerkin approximations is possible which 
may facilitate the numerical realization. The discrete derivative and average vector field collocation methods could be interpreted as such inexact realizations of the Petrov-Galerkin time-discretization.  
In case of a non-quadratic Hamiltonian, the handling of the term $\langle \H'(u),v\rangle$ in \eqref{eq:var} and the corresponding discrete equations is more subtle and may deserve further considerations.

\section*{Acknowledgements}

The work of the author was supported by the ``Excellence Initiative'' of the German Federal and State Governments via the Graduate School of Computational Engineering GSC~233 at Technische Universit\"at Darmstadt and by the German Research Foundation (DFG) via grants TRR~146, TRR~154, and Eg-331/1-1.


\begin{thebibliography}{10}

\bibitem{AlonsoValli10}
A.~Alonso~Rodr\'{i}guez and A.~Valli.
\newblock {\em Eddy current approximation of {M}axwell equations}, volume~4 of
  {\em MS\&A. Modeling, Simulation and Applications}.
\newblock Springer-Verlag Italia, Milan, 2010.
\newblock Theory, algorithms and applications.

\bibitem{Antoulas05}
A.~C. Antoulas.
\newblock {\em Approximation of large-scale dynamical systems}, volume~6 of
  {\em Advances in Design and Control}.
\newblock Society for Industrial and Applied Mathematics (SIAM), Philadelphia,
  PA, 2005.

\bibitem{BennerMehrmannSorensen05}
P.~Benner, V.~Mehrmann, and D.~C. Sorensen, editors.
\newblock {\em Dimension Reduction of Large-Scale Systems}, volume~45 of {\em
  Lect. Notes Comput. Sci. Eng.} Springer, 2005.

\bibitem{CelledoniEtAl18a}
E.~Celledoni, S.~Eidnes, B.~Owren, and T.~Ringholm.
\newblock Dissipative schemes on {R}iemannian manifolds.
\newblock {\em arXiv:1804.08104}, 2018.

\bibitem{CelledoniEtAl18b}
E.~Celledoni, S.~Eidnes, B.~Owren, and T.~Ringholm.
\newblock Energy preserving methods on {R}iemannian manifolds.
\newblock {\em arXiv:1805.07578}, 2018.

\bibitem{CelledoniEtAl17}
E.~Celledoni and E.~H. Hoiseth.
\newblock Energy-preserving and passivity-consistent numerical discretization
  of port-{H}amiltonian systems.
\newblock {\em arXive:1706.08621}, 2017.

\bibitem{CohenHairer11}
D.~Cohen and E.~Hairer.
\newblock Linear energy-preserving integrators for {P}oisson systems.
\newblock {\em BIT}, pages 91--101, 2011.

\bibitem{Egger18}
H.~Egger.
\newblock Structure preserving approximation of dissipative evolution problems,
  2018.

\bibitem{Elliott}
C.~M. Elliott.
\newblock The {C}ahn-{H}illiard model for the kinetics of phase separation.
\newblock In J.~F. Rodrigues, editor, {\em Mathematical Models for Phase Change
  Problems}, volume~88 of {\em Int. Ser. Numer. Math.} Birkh\"auser, New~York,
  1989.

\bibitem{ErnGuermond04}
A.~Ern and J.-L. Guermond.
\newblock {\em Theory and Practice of Finite Elements}.
\newblock Springer, 2004.

\bibitem{Gonzales96}
O.~Gonzalez.
\newblock Time integration and discrete {H}amiltonian systems.
\newblock {\em J. Nonlinear Sci.}, 6:449--467, 1996.

\bibitem{Hairer10}
E.~Hairer.
\newblock Energy-preserving variant of collocation methods.
\newblock {\em J. Numer. Anal. Ind. Appl. Math.}, 5:73--84, 2010.

\bibitem{HairerLubich13}
E.~Hairer and C.~Lubich.
\newblock Energy-diminishing integration of gradient systems.
\newblock {\em IMA J. Numer. Anal.}, 34:452--461, 2013.

\bibitem{HairerLubichRoche89}
E.~Hairer, C.~Lubich, and M.~Roche.
\newblock {\em The numerical solution of differential-algebraic systems by
  {R}unge-{K}utta methods}, volume 1409 of {\em Lecture Notes in Mathematics}.
\newblock Springer-Verlag, Berlin, 1989.

\bibitem{KunkelMehrmann06}
P.~Kunkel and V.~Mehrmann.
\newblock {\em Differential-algebraic equations}.
\newblock EMS Textbooks in Mathematics. European Mathematical Society (EMS),
  Z\"urich, 2006.
\newblock Analysis and numerical solution.

\bibitem{MatsuoFurihata11}
T.~Matsuo and D.~Furihata.
\newblock {\em Discrete Variational Derivative Method: a Structure-Preserving
  Numerical Method for Partial Differential Equations}.
\newblock Chapman \& Hall CRC, 2011.

\bibitem{McLachlanEtAl99}
R.~I. McLachlan, G.~R.~W. Quispel, and N.~Robidoux.
\newblock Geometric integration using discrete gradients.
\newblock {\em R. Soc. Lond. Philos. Trans. Ser. A: Math. Phys. Eng. Sci.},
  357:1021--1045, 1999.

\bibitem{Thomee06}
V.~Thom\'ee.
\newblock {\em Galerkin finite element methods for parabolic problems},
  volume~25 of {\em Springer Series in Computational Mathematics}.
\newblock Springer-Verlag, Berlin, second edition, 2006.

\end{thebibliography}

\end{document}